\documentclass[10pt]{article}
\usepackage{geometry} 
\usepackage{graphicx,url}
\usepackage{amssymb,amsmath, amsthm}
\usepackage{epstopdf}

\usepackage{enumerate}

\author{ADG}
\usepackage[T1]{fontenc}
\newtheorem{theorem}{Theorem}
\newtheorem{remark}{Remark}[section]

\pdfoutput=1
\allowdisplaybreaks
\title{Reflecting diffusions and hyperbolic Brownian motions in multidimensional spheres }
\author{ \Large{Olga Aryasova}\footnote{oaryasova@mail.ru}\\
Institute of Geophysics\\
National Academy of Sciences of Ukraine\\
 Palladin pr. 32, 03680 Kiev-142, Ukraine\\\\
\Large Alessandro De Gregorio\footnote{alessandro.degregorio@uniroma1.it},
Enzo Orsingher \footnote{enzo.orsingher@uniroma1.it}	\\
Dipartimento di Scienze Statistiche\\
``Sapienza'', University of Rome\\
P.le Aldo Moro, 5 - 00185, Rome, Italy}
\allowdisplaybreaks 
\numberwithin{equation}{section}
\begin{document}

\maketitle
\begin{abstract}

Diffusion processes $(\underline{\bf X}_d(t))_{t\geq 0}$ moving inside spheres $S_R^d \subset\mathbb{R}^d$ and reflecting orthogonally on their surfaces $\partial S_R^d$ are considered. The stochastic differential equations governing the reflecting diffusions are presented and their kernels and distributions explicitly derived. Reflection is obtained by means of the inversion with respect to the sphere $S_R^d$. The particular cases of Ornstein-Uhlenbeck process and Brownian motion are examined in detail. 

 The hyperbolic Brownian motion on the Poincar\`e half-space $\mathbb{H}_d$ is examined in the last part of the paper and its reflecting counterpart within hyperbolic spheres is studied. Finally a section is devoted to reflecting hyperbolic Brownian motion in the Poincar\`e disc $D$ within spheres concentric with $D$.
 \\
 

\end{abstract}

\section{Introduction}

The problem of the existence and of the uniqueness of the solution of stochastic differential equations with reflecting boundary conditions has been tackled by many authors. For instance Stroock and Varadhan (1971) proved the existence of the weak solution in a smooth domain, while Tanaka (1979) considered convex domains. Lions and Sznitman (1984) and Saisho (1987) provided a direct approach to the solution of stochastic differential equations with reflecting boundary conditions in nonsmooth region by exploiting the Skorohod problem. Let $D$ be a subset of $\mathbb{R}^d$, then we define a reflecting stochastic differential equations (Skorohod equation) along the normal  as 
\begin{equation*}
dX(t)=b(X(t))dt+\sigma(X(t))dB(t)+dL(t),\quad X(0)=0,
\end{equation*}
where $X(t)\in\bar D$ for all $t\geq 0$ and $L(t)$ is a bounded variation process such that
\begin{align*}
L(t)=\int_0^tn(X(s))d|L|_s,\quad |L|_t=\int_0^t1_{\partial D}(X(s))d|L|_s
\end{align*}
and $n(x)$ is the unit inward normal to $\partial D$ at $x$. More recently, the reflecting Brownian motion has been analyzed, for instance, by Bass and Hsu (1990). In the present manuscript we adopt a different approach with respect to the above mentioned papers.

 In our work we consider diffusion processes $(\underline{\bf X}_d(t))_{t\geq 0}$ in $\mathbb{R}^d$ reflecting (orthogonally) on the surface of spheres of radius $R$. We study the governing stochastic differential equations (formula \eqref{eq:sderef}) for the radial process $(\mathcal{B}_d(t))_{t\geq 0}$ which reads
\begin{equation*}
\mathcal{B}_d(t)=B_d(t){\bf 1}_{(0,R)}(B_d(t))+\frac{R^2}{ B_d(t)}{\bf 1}_{ [R,\infty)}(B_d(t)).
\end{equation*}
Reflection is thus obtained by means of spherical inversion. For the reflecting processes, we obtain the probability law 
\begin{align*}\label{eq:densd}
\overline{p}_d(r,t)&=r^{d-1}q_d(r,t)+\frac{R^{2d}}{r^{d+1}}q_d\left(\frac{R^2}{r},t\right),
\end{align*}
and the kernel
\begin{equation*}\label{eq:kernd}
\overline{q}_d(r,t)=q_d(r,t)+q_d\left(\frac{R^2}{r},t\right),
\end{equation*}
where $0<r\leq R$.

Furthermore, we examine the Neumann problem for the diffusion equation governing the law and the related kernel in some particular cases. We analyze in detail the reflecting Ornstein-Uhlenbeck process and write for its reflecting radial component $\mathcal{B}_d^{OU}(t)$ the governing stochastic equation. We also give the kernel of the reflecting Ornstein-Uhlenbeck process starting from the origin, that is
\begin{equation}\label{eq:kerouint}
 \overline{w}_d\left(r,t\right)=w_d(r,t)+w_d\left(\frac{R^2}{r},t\right)\quad 0<r\leq R,t>0,
\end{equation}
where
$$w_d(r,t)=\frac{1}{2^{\frac d2-1}\Gamma(\frac d2)\lambda^{\frac d2}}e^{-\frac{r^2}{2\lambda}},\quad \lambda=\frac{e^{2bt}-1}{2b},$$
and the related equation
$$\frac{\partial}{\partial t} \overline{w}_d\left(r,t\right)= \mathcal{L}w_d(r,t)+\tilde{\mathcal{L} }w_d\left(\frac{R^2}{r},t\right)$$
where
$$ \mathcal{L}=\frac12\frac{\partial^2}{\partial r^2}+\left(\frac{d-1}{2r}+b r\right)\frac{\partial }{\partial r},\quad \tilde{\mathcal{L} }=\frac12\frac{r^4}{R^4}\left[\frac{\partial^2 }{\partial r^2}+\left(\frac{3-d}{r}-2b\frac{R^4}{r^3}\right)\frac{\partial}{\partial r}\right].$$
Analogously, we provide the probability density function of $\mathcal{B}_d^{OU}(t)$
$$\overline{z}_d\left(r,t\right)=r^{d-1}w_d(r,t)+\frac{R^{2d}}{r^{d+1}}w_d\left(\frac{R^2}{r},t\right)$$
and the corresponding governing partial differential equation
$$\frac{\partial }{\partial t}\overline{z}_d\left(r,t\right)= \mathcal{L}^*z_d(r,t)+\tilde{\mathcal{L} }^*z_d\left(\frac{R^2}{r},t\right),$$
where $\mathcal{L}^*$, $ \tilde{\mathcal{L}}^*$ are the adjoint of $\mathcal{L}$, $ \tilde{\mathcal{L}}$ respectively.
From the analysis of the reflecting Ornstein-Uhlenbeck process, we extract, as particular case for $b=0$, the reflecting Brownian motion, which was dealt with at page 271, Problem 2, of Ito and McKean (1996).

The last section of the present work deals with the hyperbolic Brownian motion $(\eta_d(t))_{t\geq0}$ on the Poincar\`e half-space $\mathbb{H}_d$. On this process the reader can consult, for instance, Gertsenshtein and Vasiliev (1959), Karpelevich {\it et al.} (1959), Gruet (1996), (2000), Ikeda and Matsumoto (1999) and Lao and Orsingher (2007). By means of the circular inversion we obtain kernel and probability distribution of the hyperbolic reflecting Brownian motion $(\mathcal{E}_d(t))_{t\geq0}$ in spheres of  $\mathbb{H}_d$ with radius $S$. Also the stochastic integral equation governing $(\mathcal{E}_d(t))_{t\geq0}$ is presented. Since the explicit laws of $\eta_d(t),d=2,3,$ are known, we give the following formula
$$P\{\mathcal{E}_d(t)>\overline{\eta}\}=P\{\eta_d(t)>\overline{\eta}\}-P\{\eta_d(t)>S^2/\overline{\eta}\},\quad d=2,3.$$
The Millson's result permits us to generalize the above expression for $P\{\mathcal{E}_{d+2}(t)>\overline{\eta}\}$.

The reflecting hyperbolic Brownian motion in the Poincar\`e disc is analyzed in section 5.3. The simple structure of the generators of this particular case leads to the following stochastic differential equation
\begin{equation}\label{eq:poidiscint}
dD(t)=\frac{(1-D^2(t))^2}{4D(t)}dt+\frac{1-D^2(t)}{\sqrt{2}}dW(t)
\end{equation}
 From \eqref{eq:poidiscint} we are able to write the stochastic equation for the reflecting hyperbolic Brownian motion in a cirlce of radius $V<1$.

\section{Notations and preliminary results}

 For $i=1,2,...,d,$ and $d \geq 2$, let $b_i,\sigma_i$ be bounded measurable functions on $\mathbb{R}$ satisfying the following conditions.
\begin{itemize}
\item[A1]\label{condit:non-degen}
There exists $\mu>0$ such that for all $x\in\mathbb{R}$
$$
\sigma_i(x)>\mu.$$

\item[A2]\label{condit:Lip}
For all $\{x,y\}\subset\mathbb{R}$,
$$
|\sigma_i(x)-\sigma_i(y)|\leq L|x-y|,
$$
where $L$ is a positive constant.

\end{itemize}
For each  $ i=1,2,...,d$, consider a stochastic differential equation
\begin{equation}\label{equation_i}
dX_i(t)=b_i(X_i(t))dt+\sigma_i(X_i(t))dW_i(t),\quad X_i(0)=0,
\end{equation}
where $(W_i(t))_{t\geq0}$ is a standard one-dimensional Wiener process.
There exists a unique strong solution of equation \ref{equation_i} (see Zvonkin, 1974). Besides, the process $(X_i(t))_{t\geq0}$ possesses a density function $p_i(x_i,t), \ x_i\in\mathbb{R}, \ t\geq0.$ Let $\underline{\bf X}_d(t)=(X_1(t),X_2(t),...,X_d(t)),t\geq0,$ be a $d$-dimensional diffusion process where its coordinates $X_i(t)$ are independent. Therefore the probability density function of  $(\underline{\bf X}_d(t))_{t\geq0}$ has the form 
$$
p_d(\underline{{\bf x}}_d,t)=\prod_{i=1}^d p_i(x_i,t), \ x_i\in\mathbb{R}, \ i=1,2,...,d, \ t\geq0,
$$
where $\underline{{\bf x}}_d=(x_1,x_2,...,x_d)$, which is the solution for the following partial differential equation
\begin{align*}
\frac{\partial}{\partial t}p_d(\underline{{\bf x}}_d,t)=\sum_{i=1}^d\left[b_i(x_i)\frac{\partial}{\partial x_i}+\frac{\sigma_i^2(x_i)}{2}\frac{\partial ^2}{\partial x_i^2}\right]
p_d(\underline{{\bf x}}_d,t)
\end{align*}
with initial condition $p_d(\underline{{\bf x}}_d,0)=\delta(\underline{{\bf x}}_d).$

Let $B_d(t):=||\underline{\bf X}_d(t)||=\sqrt{\sum_{i=1}^d X_i^2(t)},t\geq0,$ be the radial process related to $\underline{\bf X}_d(t)$.  The process $B_d(t)$ is a generalization of the classical Bessel process and represents the main object of interest of this section. It is clear that the density function of $(B_d(t))_{t\geq0},$ is given by
\begin{align}
p_d(r,t)&:=\frac{P\{B_d(t)\in dr|B_d(0)=0\}}{dr}\notag\\
&=r^{d-1}\int_0^\pi d\theta_1\cdots\int_0^\pi d\theta_{d-2}\int_0^{2\pi}d\phi\, q_d(r,\underline\theta,t)\prod_{i=1}^{d-1}(\sin\theta_i)^{d-1-i}		\notag	\\
&:=r^{d-1}q_d(r,t)
\end{align}
with $\underline\theta=(\theta_1,...,\theta_{d-1},\phi)$ and $0<\theta_i<\pi,0<\phi<2\pi$.

Now, we take up the study of the stochastic differential equation satisfied by $B_d(t)$. Let $f(\underline{{\bf x}}_d)=||\underline{{\bf x}}_d||$, we observe that
$$\frac{\partial f(\underline{{\bf x}}_d)}{\partial x_i}=\frac{x_i}{||\underline{{\bf x}}_d||},\quad \frac{\partial^2 f(\underline{{\bf x}}_d)}{\partial x_i\partial x_j}=\frac{\delta_{ij}}{||\underline{{\bf x}}_d||}-\frac{x_ix_j}{||\underline{{\bf x}}_d||^3},$$
where $\delta_{ij}$ is the Kronecker's delta. Therefore, $f\in C^2(\mathbb{R}^d\setminus\{0\})$. In other words $f$ is not differentiable at the origin and then we cannot apply the Ito's formula to $f$. For this reason we carry out the stochastic analysis of the above process by means of arguments similar to those used for the classical Bessel process (see Karatzas and Shreve, 1998).

\begin{theorem}\label{teo:sde}
The process  $(B_d(t))_{t\geq0}$ satisfies the following stochastic differential equation
\begin{equation}\label{eq:sde}
dB_d(t)=\sum_{i=1}^d\frac{X_i(t)}{B_d(t)}\sigma_i(X_i(t))dW_i(t)+\frac{1}{2B_d(t)}\sum_{i=1}^d\left[\left\{1-\frac{X_i^2(t)}{B_d^2(t)}\right\}\sigma_i^2(X_i(t))+2X_i(t)b_i(X_i(t))\right]dt
\end{equation}
\end{theorem}

\begin{proof}
Let us define
$$Y_d(t):=B_d^2(t)=||\underline{{\bf X}}_d(t)||^2.$$

By applying the Ito's formula we obtain that
$$Y_d(t)=2\sum_{i=1}^d\int_0^tX_i(s)dX_i(s)+\sum_{i=1}^d \int_0^t\sigma_i^2(X_i(s))ds$$

Now, for $\varepsilon>0$, we consider the following function
\begin{equation}
g_\varepsilon(y)=
\begin{cases}
\frac38\sqrt{\varepsilon}+\frac{3}{4\sqrt{\varepsilon}}y-\frac{1}{8\varepsilon\sqrt{\varepsilon}}y^2,& y<\varepsilon,\\
\sqrt{y},& y> \varepsilon,
\end{cases}
\end{equation}
which is of class $C^2$ and such that $\lim_{\varepsilon\to 0}g_\varepsilon(y)=\sqrt{y}$, for all $y> 0$. Since
\begin{equation*}
\frac{\partial}{\partial x_i}g_\varepsilon(||\underline{{\bf x}}_d||^2)=
\begin{cases}
\frac{3}{2\sqrt{\varepsilon}}x_i-\frac{1}{2\varepsilon\sqrt{\varepsilon}}||\underline{{\bf x}}_d||^2x_i,& ||\underline{{\bf x}}_d||^2<\varepsilon,\\
\frac{x_i}{||\underline{{\bf x}}_d||},& ||\underline{{\bf x}}_d||^2> \varepsilon,
\end{cases}
\end{equation*}
and
\begin{equation*}
\frac{\partial^2}{\partial x_i^2}g_\varepsilon(||\underline{{\bf x}}_d||^2)=
\begin{cases}
\frac{3}{2\sqrt{\varepsilon}}-\frac{1}{2\varepsilon\sqrt{\varepsilon}}||\underline{{\bf x}}_d||^2-\frac{1}{\varepsilon\sqrt{\varepsilon}}x_i^2,& ||\underline{{\bf x}}_d||^2<\varepsilon,\\
\frac{1}{||\underline{{\bf x}}_d||}-\frac{x_i^2}{||\underline{{\bf x}}_d||^3},& ||\underline{{\bf x}}_d||^2> \varepsilon,
\end{cases}
\end{equation*}
the Ito's rule provides the following equality
\begin{equation}
g_\varepsilon(Y_d(t))=\sum_{i=1}^dA_i(\varepsilon)+\sum_{i=1}^dB_i(\varepsilon)+C(\varepsilon),
\end{equation}
where
\begin{align*}
&A_i(\varepsilon):=\int_0^t\left[\left\{\frac{3}{2\sqrt{\varepsilon}}-\frac{1}{2\varepsilon\sqrt{\varepsilon}}Y_d(s)\right\}{\bf 1}_{(0,\varepsilon)}(Y_d(s))+\frac{1}{B_d(s)}{\bf 1}_{[\varepsilon,\infty)}(Y_d(s))\right]X_i(s)dX_i(s),\\
&B_i(\varepsilon):=\frac12\int_0^t\left[\frac{1}{B_d(s)}-\frac{X_i^2(s)}{B_d^3(s)}\right]\sigma_i^2(X_i(s)){\bf 1}_{[\varepsilon,\infty)}(Y_d(s))ds,\\
&C(\varepsilon):=\int_0^t\frac{1}{4\sqrt{\varepsilon}}\left[3d-(d+2)\frac{Y_d(s)}{\varepsilon}\right]{\bf 1}_{(0,\varepsilon)}(Y_d(s))ds,
\end{align*}
(for the sake of simplicity we omit the dependency from time $t$).

Since $X_i(t)$ never attains the origin (see Bonami {\it et al.}, 1971), the Lebesgue measure of $\{0\leq s\leq t: B_d(s)=0\}$ is zero a.s. and then
$$\int_0^t\left[\frac{1}{B_d(s)}-\frac{X_i^2(s)}{B_d^3(s)}\right]\sigma_i^2(X_i(s))ds<\infty\quad \text{a.s.}$$
Hence, in force of the bounded convergence theorem we obtain that
\begin{align}
\lim_{\varepsilon\to 0}B_i(\varepsilon)&=\frac12\int_0^t\left[\frac{1}{B_d(s)}-\frac{X_i^2(s)}{B_d^3(s)}\right]\sigma_i^2(X_i(s)){\bf 1}_{[0,\infty)}(Y_d(s))ds\notag\\
&=\frac12\int_0^t\left[\frac{1}{B_d(s)}-\frac{X_i^2(s)}{B_d^3(s)}\right]\sigma_i^2(X_i(s))ds\quad \text{a.s.}
\end{align}

Furthermore, $p_i(x_i,t)$, under conditions A1 and A2, admits the upper Gaussian bound $\frac{K}{\sqrt{t}}e^{-\frac{x_i^2}{2\kappa t}}$, where $K$ and $\kappa$ are positive constants (see, for example, Portenko, 1990, Ch.2). Then we have that
\begin{align*}
0\leq E C(\varepsilon)&\leq \frac{3d}{4\sqrt{\varepsilon}}\int_0^tP\{Y_d(s)<\varepsilon\}ds\\
&\leq \frac{3d}{4\sqrt{\varepsilon}}\int_0^tP\{X_1^2(s)+X_2^2(s)<\varepsilon\}ds\\
&\leq  K\frac{3d}{4\sqrt{\varepsilon}}\int_0^t\frac{ds}{s}\int_0^{\sqrt{\varepsilon}}\rho e^{-\frac{\rho^2}{2\kappa s}}d\rho\\
&=  K\frac{3d}{4\sqrt{\varepsilon}}\int_0^{\sqrt{\varepsilon}}\rho d\rho\int_0^{t}\frac{e^{-\frac{\rho^2}{2\kappa s}}}{s}ds\\
&=(w=\rho/\sqrt{\kappa s})\\
&= K\frac{3d}{2\sqrt{\varepsilon}}\int_0^{\sqrt{\varepsilon}}\rho d\rho\int_{\rho/\sqrt{\kappa t}}^{\infty}\frac{e^{-\frac{w^2}{2}}}{w}dw
\end{align*}
and  by means of L'Hopital's rule we can conclude that $\lim_{\varepsilon\to 0}E C(\varepsilon)=0$.
Let $Z_i(t)=\int_0^t\frac{X_i(s)}{B_d(s)}dX_i(s)$, we get that
\begin{align*}
E\left\{Z_i(t)-A_i(\varepsilon)\right\}^2&=E\left\{\int_0^t{\bf 1}_{(0,\varepsilon)}(Y_d(s))\left[\frac{1}{B_d(s)}-\left(\frac{3}{2\sqrt{\varepsilon}}-\frac{1}{2\varepsilon\sqrt{\varepsilon}}Y_d(s)\right)\right]X_i(s)dX_i(s)
\right\}^2\\
&=E\left\{\int_0^t{\bf 1}_{(0,\varepsilon)}(Y_d(s))\left[\frac{1}{B_d(s)}-\left(\frac{3}{2\sqrt{\varepsilon}}-\frac{1}{2\varepsilon\sqrt{\varepsilon}}Y_d(s)\right)\right]^2X_i^2(s)\sigma_i^2(X_i(s))ds
\right\}\\
&=E\left\{\int_0^t{\bf 1}_{(0,\varepsilon)}(Y_d(s))\left[1-\frac12\sqrt{\frac{Y_d(s)}{\varepsilon}}\left(3-\frac{1}{\varepsilon}Y_d(s)\right)\right]^2\left(\frac{X_i(s)}{B_d(s)}\right)^2\sigma_i^2(X_i(s))ds\right\}\\
&\leq E\left\{\int_0^t{\bf 1}_{(0,\varepsilon)}(Y_d(s))\sigma_i^2(X_i(s))ds\right\}\\
&\leq K^2\int_0^tP\{Y_d(s)<\varepsilon\}ds
\end{align*}
which tends to zero as $\varepsilon$ goes to zero.

\end{proof}

 For $\sigma_i(X_i(t))=\sigma$ and $b_i(X_i(t))=bX_i(t)$, the process $(X_i(t))_{t\geq 0}$ becomes the Ornstein-Uhlenbeck process, namely 
\begin{equation}\label{eq:ou}
X_i(t)=\sigma\int_0^te^{-b(t-s)}dW_i(s).
\end{equation}
Clearly in this case $b_i(x)=bx$ is not a bounded function in $\mathbb{R}$. Nevertheless, the statement in Theorem \ref{teo:sde} still holds for the Ornstein-Uhlenbeck process because it is a gaussian process with variance $\frac{e^{2bt}-1}{2b}$. Then we can mimic the proof of Theorem \ref{teo:sde} and the equation \eqref{eq:sde} reduces to
\begin{equation}\label{eq:sde2}
dB_d(t)=\sigma\sum_{i=1}^d\frac{X_i(t)}{B_d(t)}dW_i(t)+\frac{1}{2B_d(t)}\left[(d-1)\sigma^2+2b\sum_{i=1}^dX_i^2(t)\right]dt.
\end{equation}

We observe that $\sum_{i=1}^d\int_0^t \frac{X_i(s)}{B_d(s)}dW_i(s):=\sum_{i=1}^d Y_i(t)$ is a standard Brownian motion. Indeed, we have that
\begin{align}
[Y_i,Y_j]_t=\int_0^t\frac{1}{B_d^2(s)}X_i(s)X_j(s)d[W_i,W_j]_s=\delta_{ij}\int_0^t\frac{1}{B_d^2(s)}X_i(s)X_j(s)ds
\end{align}
and then $\sum_{i=1}^d[Y_i]_t=t$. Furthermore, $Y_i$ is a square integrable martingale because we have that
$E(Y_i(t))^2<\infty$. Therefore, by means of L\`evy's characterization theorem we conclude that $\sum_{i=1}^d Y_i(t)$ is a standard Brownian motion that we indicate by $W(t)$. Therefore we get that
\begin{equation}\label{eq:sde2bis}
dB_d(t)=\sigma dW(t)+\left[\frac{(d-1)\sigma^2}{2B_d(t)}+bB_d(t)\right]dt.
\end{equation}

From \eqref{eq:sde2bis} by setting $b=0$ and $\sigma=1$, we obtain the stochastic differential equation for the classical Bessel process.

\section{Reflecting diffusion processes within an Euclidean sphere}\label{sect:diff}
\subsection{General case}
Let us consider a multidimensional diffusion process $(\underline{\bf X}_d(t))_{t\geq0}$ defined as in the previous section. When a sample path of the process hits the surface of the $d$-dimensional sphere $S_R^d$ with radius $R$ and centre at the origin $\underline{O}_d$ of $\mathbb{R}^d$, the sample path of the process is orthogonally reflected. This leads to a new process, namely the reflecting process $ (\underline{\mathcal{ X}}_d(t))_{t\geq0}$ moving inside the $d$-dimensional ball $S_R^d$. 

We introduce the reflecting diffusion process by means of the circular inversion of a point with respect to the circle. If we consider a point $x$ inside $S_R^d$, having polar coordinates equal to $(r,\theta)$, we can find another point $y$ in the space $\mathbb{R}^d$ with polar coordinates given by $(r',\theta)$ ($R<r'$), such that $rr'=R^2$. The point $b$ is called the inverse point of $a$ with respect to $S_R^d$.  Therefore if we indicate by $\mathcal{B}_d(t)=||\underline{\mathcal{ X}}_d(t)||,t\geq0,$ the radial component of $(\underline{\bf X}_d(t))_{t\geq0}$, we have that 
\begin{equation}\label{eq:dbessref}
\mathcal{B}_d(t)=
\begin{cases}
B_d(t),& B_d(t)\in(0,R),\\
\frac{R^2}{ B_d(t)},& B_d(t)\in[R,\infty),
\end{cases}
\end{equation}
or equivalently
\begin{equation}
\mathcal{B}_d(t)=B_d(t){\bf 1}_{(0,R)}(B_d(t))+\frac{R^2}{ B_d(t)}{\bf 1}_{ [R,\infty)}(B_d(t)).
\end{equation}

 The process \eqref{eq:dbessref}, has two components: the first one in the interval $(0,R)$ is related to the Bessel process without reflection and the second component concerns the reflection of the sample path crossing the surface of the sphere. In other words the excursions of sample paths of  $ (\underline{\bf{ X}}_d(t))_{t\geq0}$ outside the sphere are mapped inside $S_R^d$ by using the circular inversion  (see Figure \ref{fig1}). When the sample path of the non-reflecting process tends to go far from the origin the reflected trajectory is located near the origin of $\mathbb{R}^d$.
\begin{figure}[t]
\begin{center}
\includegraphics[angle=0,width=0.8\textwidth]{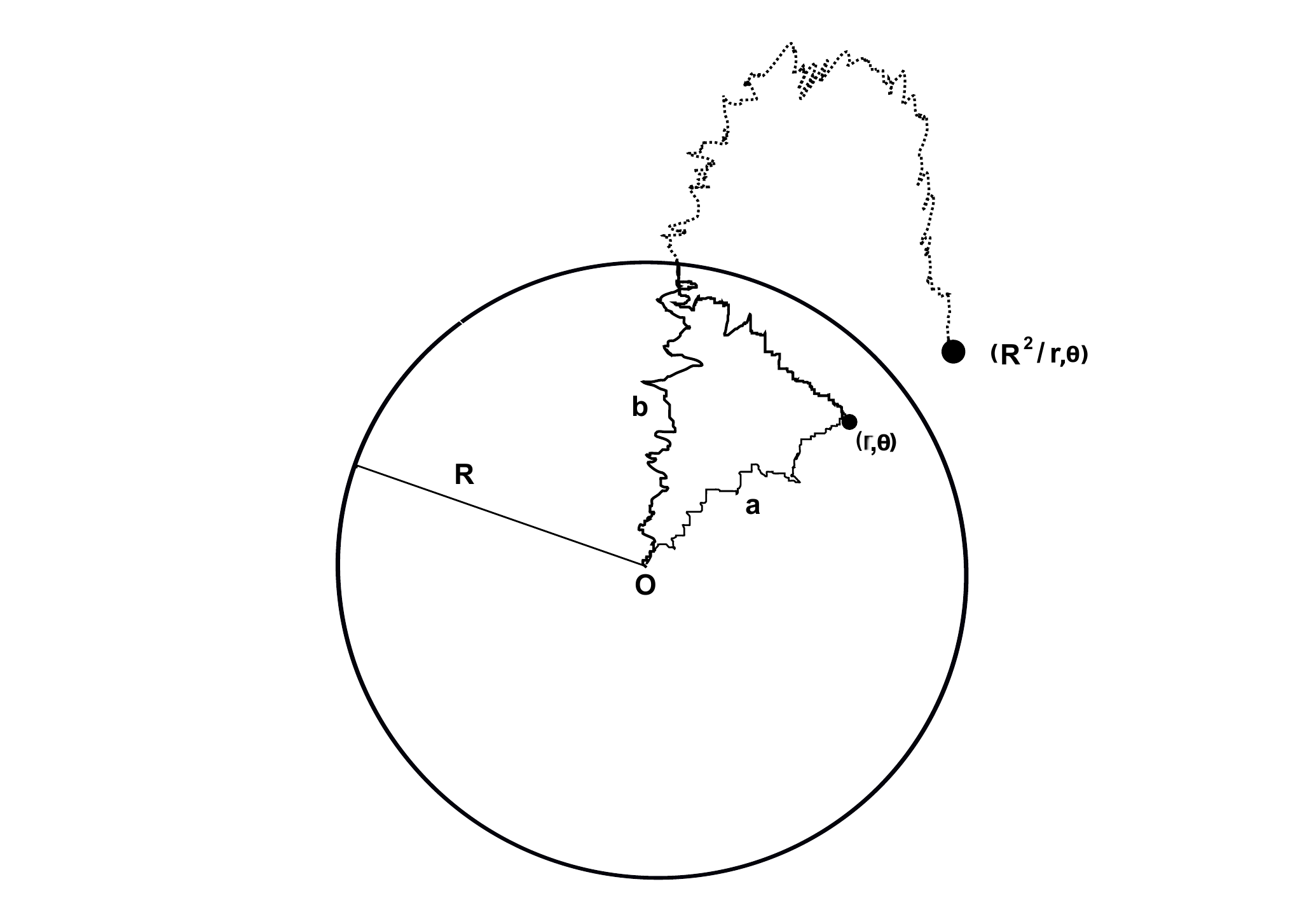}
\caption{In figure two sample paths a and b are depicted. The trajectory wandering outside the circle $S_R^d$ (indicated by the dotted line) is reflected inside $S_R^d$ by circular inversion.}\label{fig1}
\end{center}
\end{figure}

\begin{theorem}\label{teo:kerndenssde}
The kernel of $(\mathcal{B}_d(t))_{t\geq0}$ is given by
\begin{equation}\label{eq:kernd}
\overline{q}_d(r,t)=q_d(r,t)+q_d\left(\frac{R^2}{r},t\right),
\end{equation}
satisfying Neumann condition $$\frac{\partial }{\partial r}\overline q_d(r,t)\Big|_{r=R}=0,$$ 
while the density function becomes
\begin{align}\label{eq:densd}
\overline{p}_d(r,t)&=r^{d-1}q_d(r,t)+\frac{R^{2d}}{r^{d+1}}q_d\left(\frac{R^2}{r},t\right),
\end{align}
where $0<r\leq R$.
\end{theorem}

\begin{proof}

In order to obtain the kernel and density function of $(\mathcal{B}_d(t))_{t\geq 0},$ we observe that is associated to each point $\underline{{\bf x}}_d$, inside the support of the process, either a part of probability deriving by the sample path that directly reach the point $\underline{{\bf x}}_d$ and the component of probability related to the sample paths moving outside the sphere $S_R^d$ (see Figure \ref{fig1}).
 
The first term in \eqref{eq:kernd} is the kernel of the non-reflecting process $B_d(t)$, while the second term is obtained by considering the inversion with respect to the surface $S_R^d$ of radius $R$ which implies $r'=R^2/r,0<r<R$.  It is easy to verify that $ \overline q_d(r,t)$ satisfies the reflection condition
$$\frac{\partial }{\partial r}\overline q_d(r,t)\Big|_{r=R}=0$$

We observe that the infinitesimal volume element related to the point $(r',\theta)$ is written as $$d\mathcal{V}_d=(r')^{d-1}dr'\prod_{i=1}^{d-1}(\sin\theta_i)^{d-1-i}d\theta_i,$$
where $r'=\frac{R^2}{r}$ and $dr'=\frac{R^2}{r^2}dr$, which imply that
$$d\mathcal{V}_d=\frac{R^{2d}}{r^{d+1}}dr\prod_{i=1}^{d-1}(\sin\theta_i)^{d-1-i}d\theta_i,$$
with $\theta_i\in[0,\pi],i=1,...,d-2$ and $\theta_{d-1}=\phi\in[0,2\pi]$. Then, the density function of $(\mathcal{B}_d(t))_{t\geq 0},$ is given by
\begin{align}
\overline{p}_d(r,t)&=r^{d-1}q_d(r,t)+\frac{R^{2d}}{r^{d+1}}q_d\left(\frac{R^2}{r},t\right),
\end{align}
where $0<r\leq R$.
Furthermore, we have that $\overline p_d(r,t)$ integrates to 1 as we show below
\begin{align*}
\int_{0}^R\overline{p}_d(r,t)dr
&=\int_0^Rr^{d-1}q_d(r,t)dr+R^{2d}\int_0^R\frac{1}{r^{d+1}}q_d\left(\frac{R^2}{r},t\right)dr\\
&=(y=R^2/r)\\
&=\int_0^\infty r^{d-1}q_d(r,t)dr\\
&=1
\end{align*}
This result concludes the proof.

\end{proof}

\begin{remark}
Since the kernel $q_d(r,t)$ of the diffusion is bounded by $\frac{K}{\sqrt{t}}e^{-\frac{r^2}{2\kappa t}}$, where $K$ and $\kappa$ are positive constants, the component related to the reflection, that is $q_d\left(\frac{R^2}{r},t\right)$, tends to zero as $R\to \infty$. Then
$$\overline{q}_d(r,t)\to r^{d-1}q_d(r,t)$$
or equivalently $\mathcal{B}_d(t)\stackrel{d}{\to} B_d(t),$ as $R\to \infty$.
\end{remark}

By means of Theorem \ref{teo:sde}, we are able to provide the stochastic differential equation governing the reflecting process $(\mathcal{B}_d(t))_{t\geq0}.$ 

\begin{theorem}\label{itobessel} The process $(\mathcal{B}_d(t))_{t\geq0}$ solves the following stochastic differential equation
\small{\begin{align}\label{eq:sderef}
&\mathcal{B}_d(t)\notag\\
&=\sum_{i=1}^d\int_0^t\left[{\bf 1}_{(0,R)}(B_d(s))-\frac{R^2}{B_d^2(s)}{\bf 1}_{ [R,\infty)}(B_d(s))\right]\frac{X_i(s)}{B_d(s)}\sigma_i(X_i(s))dW_i(s)\notag\\
&\quad+\sum_{i=1}^d\int_0^t\frac{1}{2B_d(
s)}\Bigg\{\left[{\bf 1}_{(0,R)}(B_d(s))-\frac{R^2}{B_d^2(s)}{\bf 1}_{ [R,\infty)}(B_d(s))\right]\left[\left\{1-\frac{X_i^2(s)}{B_d^2(s)}\right\}\sigma_i^2(X_i(s))+2X_i(s)b_i(X_i(s))\right]\notag\\
&\quad+R^2\frac{X_i^2(s)}{B_d^5(s)}\sigma_i^2(X_i(s)){\bf 1}_{ [R,\infty)}(B_d(s))\Bigg\}ds-L_t(R),
\end{align}}
$L_t(R)$ is the local time of $\mathcal{B}_d(t)$ at the point $R$, defined by
$$L_t(R)=\lim_{\varepsilon \to 0}\frac1\varepsilon\int_0^t{\bf 1}_{[R,R+\varepsilon)}(\mathcal{B}_d(s))ds.$$
\end{theorem}
\begin{proof}
We observe that the function
$$g(x)=\begin{cases}
x,& x\in(0,R),\\
\frac{R^2}{x},& x\in[R,\infty),
\end{cases}$$
does not admit second derivative at the point $R$. For this reason we cannot apply the Ito formula but we will use an extension for convex functions, i.e. the Meyer-Ito formula (see Theorem 70 pag. 214, Protter, 2004). 

First of all we prove that the second generalized derivative of $g$ is equal to
$$g''(x)=2\frac{R^2}{x^3}{\bf 1}_{ [R,\infty)}(x)-2\delta(x-R).
$$

Let $\phi:(0,\infty)\to\mathbb{R}$ be a test function which is infinitely differentiable and is identically zero outside some bounded interval of $(0,\infty)$. We have that
\begin{align*}
\int_{0}^{+\infty} g''(x)\phi(x)dx&=\int_{0}^{+\infty}g(x)\phi''(x)dx\\
&=\int_0^Rx\phi''(x)dx+\int_R^{+\infty}\frac{R^2}{x}\phi''(x)dx\\
&=\int_{0}^{+\infty}\left[\frac{2R^2}{x^3}{\bf 1}_{[R,\infty)}(x)-2\delta(x-R)\right]\phi(x)dx.
\end{align*}
The equation \eqref{eq:sde} implies that $(B_d(t))_{t\geq 0}$ is a semimartingale and then we can apply the Meyer-Ito formula to $\mathcal{B}_d(t)=g(B_d(s))$ as follows
\begin{align*}
\mathcal{B}_d(t)&=\int_0^tD_-g(B_d(s))dB_d(s)+\frac12\int_{-\infty}^{+\infty}L_t(a)g''(a)da\\
&=\int_0^t\left[{\bf 1}_{(0,R)}(B_d(s))-\frac{R^2}{B_d^2(s)}{\bf 1}_{ [R,\infty)}(B_d(s))\right]dB_d(s)+\int_{-\infty}^{+\infty}L_t(a)\left[\frac{R^2}{a^3}{\bf 1}_{ [R,\infty)}(a)-\frac{R^2}{a^2}\delta(a-R)\right]da\\
&=\int_0^t\left[{\bf 1}_{(0,R)}(B_d(s))-\frac{R^2}{B_d^2(s)}{\bf 1}_{ [R,\infty)}(B_d(s))\right]\\
&\quad\times\left\{\sum_{i=1}^d\frac{X_i(s)}{B_d(s)}\sigma_i(X_i(s))dW_i(s)+\frac{1}{2B_d(s)}\sum_{i=1}^d\left[\left\{1-\frac{X_i^2(s)}{B_d^2(s)}\right\}\sigma_i^2(X_i(s))+2X_i(s)b_i(X_i(s))\right]ds\right\}\\
&\quad+\sum_{i=1}^d\int_0^t\frac{R^2}{B_d^5(s)}X_i^2(s)\sigma_i^2(X_i(s)){\bf 1}_{ [R,\infty)}(B_d(s))ds-L_t(R),
\end{align*}
where in the last step we have used the cross variation process $$[B_d,B_d]_t=\sum_{i=1}^d\int_0^t\frac{X_i^2(s)}{B_d^2(s)}\sigma_i^2(X_i(s))ds$$ and Corollary 1, pag.216, Protter (2004), which provides the following equality
\begin{align*}
\int_{-\infty}^{+\infty}L_t(a)\frac{R^2}{a^3}{\bf 1}_{ [R,\infty)}(a)da&=\int_0^t\frac{R^2}{B_d^3(s)}{\bf 1}_{ [R,\infty)}(B_d(s))d[B_d,B_d]_s\\
&=\sum_{i=1}^d\int_0^t\frac{R^2}{B_d^5(s)}X_i^2(s)\sigma_i^2(X_i(s)){\bf 1}_{ [R,\infty)}(B_d(s))ds.
\end{align*}
\end{proof}

\subsection{Reflecting Ornstein-Uhlenbeck process}
Let $X_i(t)$ be the classical Ornstein-Uhlenbeck process \eqref{eq:ou} and let $(B_d(t))_{t\geq0}$ be the radial Ornstein-Uhlenbeck process which is solution for equation \eqref{eq:sde2bis}. The transition function for $B_d(t)$ is given by

$$z_d(r,r_0,t,s)=\left(\frac{r}{r_0e^{b(t-s)}}\right)^{\frac d2}\frac{r_0e^{b(t-s)}}{\lambda}\exp{\left\{-\frac{1}{2\lambda}(r^2+r_0^2e^{b(t-s)})\right\}}I_{\frac d2-1}\left(\frac{rr_0e^{b(t-s)}}{\lambda}\right)$$
with $0<s<t$ and $\lambda=\frac{e^{2b t}-1}{2b}$ (see Eie, 1983). For $s$ and $r_0$ which tend to zero the above expression becomes
\begin{equation}\label{eq:densou}
z_d(r,t)=r^{d-1} w_d(r,t)
\end{equation}
with 
\begin{equation}\label{eq:kerou}
w_d(r,t)=\frac{1}{2^{\frac d2-1}\Gamma(\frac d2)\lambda^{\frac d2}}e^{-\frac{r^2}{2\lambda}}.
\end{equation}

The reflecting Ornstein-Uhlenbeck process is a particular case of the process $(\underline{\mathcal{X}}_d(t))_{t\geq0},$ with $b_i(x)=b x_i$ and $\sigma_i(x)=1$, $i=1,2,...,d$. By Theorem \ref{itobessel} and \eqref{eq:sde2bis}, we obtain that the distance from the origin of the reflecting Ornstein-Uhlenbeck process, that we indicate by $(\mathcal{B}_d^{OU}(t))_{t\geq0}$, satisfies the following stochastic integral equation
\begin{align}\label{eq:sdeou}
\mathcal{B}_d^{OU}(t)&=\int_0^t\left[{\bf 1}_{ (0,R)}(B_d(s))-\frac{R^2}{B_d^2(s)}{\bf 1}_{ [R,\infty)}(B_d(s))\right]dW(s)	\notag\\
&\quad+\int_0^t\left[{\bf 1}_{ (0,R)}(B_d(s))-\frac{R^2}{B_d^2(s)}{\bf 1}_{ [R,\infty)}(B_d(s))\right]\left(\frac{d-1}{2B_d(s)}+b B_d(s)\right)ds\notag\\
&\quad+\int_0^t\frac{R^2}{B_d^3(s)}{\bf 1}_{ [R,\infty)}(B_d(s))ds-L_t(R).
\end{align}
Furthermore,  the density function $\overline p_d$ and the kernel $\overline q_d$ of $(\mathcal{B}_d^{OU}(t))_{t\geq0}$ are obtained by plugging into \eqref{eq:densd} and \eqref{eq:kernd} the expressions \eqref{eq:densou} and \eqref{eq:kerou} respectively. 

For the reflecting Ornstein-Unhlenbeck process it is possible to obtain the partial differential equation satisfied by $\overline z_d(r,t)$ and $\overline w_d(r,t)$. The next theorem provides these results.

\begin{theorem}\label{genou}
The kernel function $$\overline{w}_d(r,t)=w_d(r,t)+w_d\left(\frac{R^2}{r},t\right)$$ of the reflecting radial Ornstein-Uhlenbeck process  $(\mathcal{B}_d^{OU}(t))_{t\geq0},$ is the solution to the following Cauchy problem
\begin{equation}\label{eq:kerneqou}
\begin{cases}
\frac{\partial}{\partial t} \overline{w}_d\left(r,t\right)= \mathcal{L}w_d(r,t)+\tilde{\mathcal{L} }w_d\left(\frac{R^2}{r},t\right),\quad 0<r\leq R,\\
\overline{w}_d\left(r,0\right)=\delta(r),\\
\frac{\partial}{\partial r}\overline{w}_d\left(r,t\right)\Big|_{r=R}=0,
\end{cases}
\end{equation}
where
$$ \mathcal{L}=\frac12\frac{\partial^2}{\partial r^2}+\left(\frac{d-1}{2r}+b r\right)\frac{\partial }{\partial r},\quad \tilde{\mathcal{L} }=\frac12\frac{r^4}{R^4}\left[\frac{\partial^2 }{\partial r^2}+\left(\frac{3-d}{r}-2b\frac{R^4}{r^3}\right)\frac{\partial}{\partial r}\right].$$
Furthermore, the density function $$\overline{z}_d(r,t)=r^{d-1}w_d(r,t)+\frac{R^{2d}}{r^{d+1}}w_d\left(\frac{R^2}{r},t\right)$$ is the solution to the following Cauchy problem
\begin{equation}\label{eq:denseqou}
\begin{cases}
\frac{\partial }{\partial t}\overline{z}_d\left(r,t\right)= \mathcal{L}^*z_d(r,t)+\tilde{\mathcal{L} }^*z_d\left(\frac{R^2}{r},t\right),\quad 0<r\leq R,\\
\overline{z}_d\left(r,0\right)=\delta(r),\\
\frac{\partial}{\partial t}\overline{z}_d\left(r,t\right)\Big|_{t=0}=0,
\end{cases}
\end{equation}
where
$$ \mathcal{L}^*=\frac12\frac{\partial^2}{\partial r^2}-\left(\frac{d-1}{2r}+b r\right)\frac{\partial}{\partial r}+\left(\frac{d-1}{2r^2}-b \right),$$$$ \tilde{\mathcal{L}}^*= \frac12\frac{r^4}{R^4}\left[\frac{\partial^2 }{\partial r^2}+\left(\frac{d+1}{r}+2b\frac{R^4}{r^3}\right)\frac{\partial }{\partial r}+\left(\frac{d-1}{r^2}-2\frac{R^4}{r^4}b\right)\right].$$
\end{theorem}
\begin{proof}

The kernel $w_d(r,t)$ satisfies the following partial differential equation
\begin{align*}
\frac{\partial }{\partial t}w_d(r,t)&=\mathcal{L}w_d(r,t)\\
&=\frac12\frac{\partial^2 w_d(r,t)}{\partial r^2}+\left(\frac{d-1}{2r}+b r\right)\frac{\partial w_d(r,t)}{\partial r}
\end{align*}
while the density function $z_d(r,t)$ is solution to the following partial differential equation
\begin{align}\label{eq:pdeou}
\frac{\partial }{\partial t}z_d(r,t)&=\mathcal{L}^*z_d(r,t)\notag\\
&=\frac12\frac{\partial^2 }{\partial r^2}z_d(r,t)-\frac{\partial }{\partial r}\left[\left(\frac{d-1}{2r}+b r\right)z_d(r,t)\right]\notag\\
&=\frac12\frac{\partial^2 }{\partial r^2}z_d(r,t)-\left(\frac{d-1}{2r}+b r\right)\frac{\partial }{\partial r}z_d(r,t)+\left(\frac{d-1}{2r^2}-b \right)z_d(r,t)
\end{align}
with $r>0.$ Therefore, if we consider $r'=R^2/r> R$ (with $0<r\leq  R$) we have that
\begin{equation}\label{eq:inv}
\frac{\partial }{\partial t}w_d(r',t)=\frac12\left[\frac{\partial^2}{\partial r'^2}w_d(r',t)+\left(\frac{d-1}{r'}+2b r'\right)\frac{\partial}{\partial r'}w_d(r',t)\right].
\end{equation}

It is easy to see that
$$\frac{\partial}{\partial r}=-\frac{R^2}{r^2}\frac{\partial}{\partial r'},\quad \frac{\partial^2}{\partial r^2}=\frac{2R^2}{r^3}\frac{\partial}{\partial r'}+\frac{R^4}{r^4}\frac{\partial^2}{\partial r'^2},$$
or equivalently

$$\frac{\partial}{\partial r'}=-\frac{r^2}{R^2}\frac{\partial}{\partial r},\quad \frac{\partial^2}{\partial r'^2}=\frac{r^4}{R^4}\left(\frac{\partial^2}{\partial r^2}+\frac2r\frac{\partial}{\partial r}\right).$$
 Plugging into \eqref{eq:inv} the above expressions, we obtain the following equality
\begin{align*}
\frac{\partial }{\partial t}w_d\left(\frac{R^2}{r},t\right)&= \tilde{\mathcal{L} }w_d\left(\frac{R^2}{r},t\right)\\
&=\frac12\frac{r^4}{R^4}\left[\frac{\partial^2 }{\partial r^2}w_d\left(\frac{R^2}{r},t\right)+\left(\frac{3-d}{r}-2b\frac{R^4}{r^3}\right)\frac{\partial }{\partial r}w_d\left(\frac{R^2}{r},t\right)\right],
\end{align*}
which implies equation \eqref{eq:kerneqou}. It is not hard to check that the initial and the boundary conditions appearing in the problem \eqref{eq:kerneqou} hold.

Now, we focus our attention on the density $z_d$. By considering \eqref{eq:pdeou}, the same substitutions adopted in the previous calculations lead to the following equality
\begin{align*}
\frac{\partial }{\partial t}z_d\left(\frac{R^2}{r},t\right)&= \frac12\frac{r^4}{R^4}\left[\frac{\partial^2 }{\partial r^2}z_d\left(\frac{R^2}{r},t\right)+\left(\frac{d+1}{r}+2b\frac{R^4}{r^3}\right)\frac{\partial }{\partial r}z_d\left(\frac{R^2}{r},t\right)\right]\\
&\quad+\left(\frac{r^2}{R^4}\frac{d-1}{2}-b\right)z_d\left(\frac{R^2}{r},t\right)
\end{align*}
and therefore the result \eqref{eq:denseqou} follows.

\end{proof}

\begin{remark}

For $d=2$, we are able to obtain the distribution function of $(\mathcal{B}_2^{OU}(t))_{t\geq0}$ as follows

\begin{align*}
P\left\{\mathcal{B}_2^{OU}(t)<R'\right\}&=\int_0^{R'}rw_2(r,t)dr+R^4\int_0^{R'}\frac{1}{r^{3}}w_2\left(\frac{R^2}{r},t\right)dr\\
&= \int_0^{R'}rw_2(r,t)dr+ \int_{R^2/R'}^\infty rw_2(r,t)dr\\
&=1-e^{-\frac{R'^2}{2\lambda}}+e^{-\frac{R^4}{2R'^2\lambda}}
\end{align*}
that for small values of $R'$ becomes
\begin{align*}
P\left\{\mathcal{B}_2^{OU}(t)<R'\right\}&\sim 1-\left(1-\frac{R'^2}{2\lambda}\right)=\frac{R'^2}{2\lambda}.
\end{align*}
Furthermore, we observe that
$$P\left\{R_1<\mathcal{B}_2^{OU}(t)<R_2\right\}=e^{-\frac{R_1^2}{2\lambda}}-e^{-\frac{R_2^2}{2\lambda}}+e^{-\frac{R^4}{2R_2^2\lambda}}-e^{-\frac{R^4}{2R_1^2\lambda}}.$$
\end{remark}

\begin{remark}
We observe that the kernel for the reflecting Ornstein-Uhlenbeck process in spherical coordinates admits the following representation
\begin{align*}
\overline{w}_d\left(r,\underline\theta,t\right)&=\frac{\Gamma(d/2)}{2^{\frac d2}\pi^{\frac d2}}\overline{w}_d\left(r,t\right)\\
&=\frac{1}{2^{\frac d2}\pi^{\frac d2}\lambda^{\frac d2}}\left[e^{-\frac{r^2}{2\lambda}}+e^{-\frac{R^4}{2\lambda r^2}}\right]
\end{align*}
that in cartesian coordinates becomes
\begin{align*}
\overline{w}_d\left(\underline{{\bf x}}_d,t\right)=\frac{1}{2^{\frac d2}\pi^{\frac d2}\lambda^{\frac d2}||\underline{{\bf x}}_d||^{d-1}}\left[e^{-\frac{||\underline{{\bf x}}_d||^2}{2\lambda}}+e^{-\frac{R^4}{2\lambda ||\underline{{\bf x}}_d||^2}}\right].
\end{align*}
Analogously, for the density function we have that
\begin{align*}
\overline{z}_d\left(\underline{{\bf x}}_d,t\right)=\frac{1}{2^{\frac d2}\pi^{\frac d2}\lambda^{\frac d2}}\left[e^{-\frac{||\underline{{\bf x}}_d||^2}{2\lambda}}+\frac{R^{2d}}{||\underline{{\bf x}}_d||^{2d}}e^{-\frac{R^4}{2\lambda ||\underline{{\bf x}}_d||^2}}\right].
\end{align*}
\end{remark}
\begin{remark}

Let $m\geq 1$, we have that
\begin{align*}
E\left[\mathcal{B}_d^{OU}(t)\right]^m&=\int_0^Rr^{d-1+m}w_d(r,t)dr+\int_0^R\frac{R^{2d}}{r^{d+1-m}}w_d\left(\frac{R^2}{r},t\right)dr\\
&= (y=R^2/r)\\
&=\int_0^Rr^{d-1+m}w_d(r,t)dr+R^{2m}\int_R^\infty y^{d-1-m}w_d(y,t)dy\\
&=\frac{1}{2^{\frac d2-1}\Gamma(\frac d2)}\left\{\int_0^Rr^{d-1+m}\frac{1}{\lambda^{d/2}}e^{-\frac{r^2}{2\lambda}}dr+R^{2m}\int_R^\infty r^{d-1-m}\frac{1}{\lambda^{d/2}}e^{-\frac{r^2}{2\lambda}}dr\right\}\\
&=(w=r^2/2\lambda)\\
&=\frac{1}{2^{\frac d2-1}\Gamma(\frac d2)}\left\{2^{\frac{d+m}{2}-1}\lambda^{m/2}\int_0^{R^2/2\lambda}e^{-w}w^{\frac{d+m}{2}-1}dw+\frac{R^{2m}2^{\frac{d-m}{2}-1}}{\lambda^{m/2}}\int_{R^2/2\lambda}^\infty e^{-w}w^{\frac{d-m}{2}-1}dw\right\}\\
&=\frac{1}{\Gamma(\frac d2)}\left\{(2\lambda)^{m/2}\gamma\left(\frac{d+m}{2},\frac{R^2}{2\lambda}\right)+\frac{R^{2m}}{(2\lambda)^{m/2}}\Gamma\left(\frac{d-m}{2},\frac{R^2}{2\lambda}\right)\right\},
\end{align*}
where $\gamma(\cdot,\cdot)$ is the incomplete gamma function and $\Gamma(\cdot,\cdot)$ represents its complement.
For $d=3$, the mean value of $(\mathcal{B}_d^{OU}(t))_{t\geq 0}$ becomes
\begin{align*}
E\left[\mathcal{B}_3^{OU}(t)\right]&=\frac{2}{\sqrt{\pi}}\left\{\sqrt{2\lambda}\int_0^{R^2/2\lambda}e^{-w}wdw+\frac{R^{2}}{\sqrt{2\lambda}}\int_{R^2/2\lambda}^\infty e^{-w}dw\right\}\\
&=\frac{2\sqrt{2\lambda}}{\sqrt{\pi}}(1-e^{-\frac{R^2}{2\lambda}}).
\end{align*}

\end{remark}

\subsection{Brownian motion case}
From the reflecting Ornstein-Uhlenbeck process we derive as particular case the reflecting Brownian motion. It is well-known that for $b=0$ the Ornstein-Uhlenbeck process reduces to the standard Brownian motion and consequently $B_d(t)$ becomes the classical $d$-dimensional Bessel process. Let $f_d(r,t)$ and $g_d(r,t)$ be the kernel function and the density function, respectively, of $(B_d(t))_{t\geq0}$, that is
$$ g_d(r,t)=r^{d-1}f_d(r,t)=\frac{r^{d-1}}{2^{\frac d2-1}\Gamma(\frac d2)t^{\frac d2}}e^{-\frac{r^2}{2t}}.$$
Let us indicate by $(\mathcal{B}_d^{W}(t))_{t\geq0}$ the reflecting Bessel process. Theorem \ref{teo:kerndenssde} permits us to write explicit the kernel function $\overline{f}_d(r,t)$ and the density law $\overline{g}_d(r,t)$ of the reflecting Bessel process, while \eqref{eq:sdeou}, for $b=0$, represents the stochastic differential equation governing $(\mathcal{B}_d^{W}(t))_{t\geq0}$.

By setting $b=0$ in the statements of Theorem \ref{genou}, we obtain that
\begin{equation}
\frac{\partial}{\partial t} \overline{f}_d\left(r,t\right)= \mathcal{M}f_d(r,t)+\tilde{\mathcal{M} }f_d\left(\frac{R^2}{r},t\right),
\end{equation}
where
$$ \mathcal{M}=\frac12\frac{\partial^2}{\partial r^2}+\left(\frac{d-1}{2r}\right)\frac{\partial }{\partial r},\quad \tilde{\mathcal{M} }=\frac12\frac{r^4}{R^4}\left[\frac{\partial^2 }{\partial r^2}+\left(\frac{3-d}{r}\right)\frac{\partial}{\partial r}\right].$$
Analogously, the density function $\overline{g}_d(r,t)$ is solution to the following p.d.e.
\begin{equation}
\frac{\partial }{\partial t}\overline{g}_d\left(r,t\right)= \mathcal{M}^*g_d(r,t)+\tilde{\mathcal{M} }^*g_d\left(\frac{R^2}{r},t\right),
\end{equation}
where
$$ \mathcal{M}^*=\frac12\frac{\partial^2}{\partial r^2}-\left(\frac{d-1}{2r}\right)\frac{\partial}{\partial r}+\left(\frac{d-1}{2r^2} \right),\quad \tilde{\mathcal{M}}^*=\frac12\frac{r^4}{R^4}\left[\frac{\partial^2 }{\partial r^2}+\left(\frac{d+1}{r}\right)\frac{\partial }{\partial r}+\left(\frac{d-1}{r^2}\right)\right].$$

\begin{remark}

We observe that the Laplace transform of the density function of the process $(\mathcal{B}_d^{W}(t))_{t\geq 0}$ becomes
\begin{align*}
L\{\overline{g}_d(r,t)\}(s)&=\int_0^\infty e^{-st}\overline{g}_d(r,t)dt\\
&=\frac{r^{d-1}}{2^{\frac d2-1}\Gamma(\frac d2)}\int_0^\infty e^{-st}\frac{e^{-r^2/2t}}{t^{\frac d2}}dt+\frac{R^{2d}}{2^{\frac d2-1}\Gamma(\frac d2)r^{d+1}}\int_0^\infty e^{-st}\frac{e^{-R^4/2tr^2}}{t^{\frac d2}}dt\\
&=\frac{2^{-\frac {d}{4}+\frac32}s^{\frac{d-2}{4}}r^{\frac d2}}{\Gamma(\frac d2)}\left[K_{\frac d2-1}(\sqrt{2s}r)+\frac{R^{d-2}}{r^2}K_{\frac d2-1}\left(\frac{\sqrt{2s}R^2}{r}\right)\right],\quad s>0,
\end{align*}
where
$$K_\nu(z)=\frac12\left(\frac z2\right)^\nu\int_0^\infty e^{-t-\frac{z^2}{4t}}t^{-\nu-1}dt$$
with $\nu\in \mathbb{R}$, is the Bessel function of imaginary argument. Since
$$K_{\frac12}(x)=\left(\frac{\pi}{2x}\right)^{\frac12}e^{-x},$$
for  $d=3$, the following equality holds
\begin{align*}
L\{\overline{g}_3(r,t)\}(s)=\frac{1}{2^{\frac14}}\left[re^{-\sqrt{2s}r}+e^{-\frac{\sqrt{2s}R^2}{r}}\right].
\end{align*}

\end{remark}

\section{Reflecting Hyperbolic Brownian motions}

\subsection{A brief introduction to the hyperbolic framework}
Let us start this section summing up basic definitions and well-known results on the hyperbolic Brownian motion (see Gruet, 1996, 2000 and Lao and Orsingher, 2007).
The $d$-dimensional hyperbolic Brownian motion, with $d\geq2$, is a diffusion process defined on the Poincar\`e upper-half space
$\mathbb{H}_d=\{\underline{{\bf x}}_d:\underline{{\bf x}}_{d-1}\in\mathbb{R}^{d-1}, x_d>0\}$
with metric
$$ds^2=\frac{\sum_{i=1}^d dx_i^2}{x_d^2}$$ The transition density $p_{d}(\underline{{\bf x}}_d,t)$ of the hyperbolic Brownian motion satisfies the following heat-type equation
$$\frac{\partial }{\partial t}p_{d}(\underline{{\bf x}}_d,t)=\frac12\left[x_d^2\sum_{i=1}^d\frac{\partial^2}{\partial x_i^2}-(d-2)x_d\frac{\partial}{\partial x_i}\right]p_{d}(\underline{{\bf x}}_d,t)$$
subject to the initial condition
$$q_d(\underline{{\bf x}}_d,0)=\delta(x_1)\delta(x_2)\cdots\delta(x_{d-1})\delta(x_d-1).$$

The hyperbolic distance $\eta$ from the origin $\underline{O}_d=(\underline0_{d-1},1)$ of $\mathbb{H}_d$ is expressed by means of the following equality
$$\cosh\eta=\frac{\sum_{i=1}^{d-1} x_i^2+x_d^2+1}{2x_d}.$$

The main object of interest in the analysis of the hyperbolic Brownian motion is the process $(\eta_d(t))_{t\geq0},$ which represents the hyperbolic distance process of the Brownian motion in $\mathbb{H}_d$. The kernel $u_d(\eta,t)$ of $(\eta_d(t))_{t\geq0}$ is the solution to the following Cauchy problem
\begin{equation}
\begin{cases}
\frac{\partial}{\partial t} u_d=\mathcal {P}u_d,\\
u_d(\eta,0)=\delta(\eta),
\end{cases}
\end{equation}
where $\mathcal {P}= \frac{1}{2\sinh^{d-1}\eta}\frac{\partial}{\partial\eta}\left(\sinh^{d-1}\eta\frac{\partial}{\partial\eta}\right)$.
For $d=2$, we have that
$$u_2(\eta,t)=\frac{e^{-t/4}}{\sqrt{\pi}(\sqrt{2t})^3}\int_\eta^\infty\frac{\phi e^{-\phi^2/4t}}{\sqrt{\cosh\phi-\cosh\eta}}d\phi,\quad \eta>0,$$
and for $d=3$
$$u_3(\eta,t)=\frac{e^{-t}}{2\sqrt{\pi}t^{3/2}}\frac{\eta e^{-\eta^2/4t}}{\sinh\eta},\quad \eta>0.$$
In general we have that
$$u_{d+2}(\eta,t)=-\frac{ e^{-dt}}{2\pi\sinh\eta}\frac{\partial}{\partial \eta}u_d(\eta,t)$$
with $d=1,2,...$
(see Grigor'yan and Noguchi, 1998).

We observe that the volume element in $\mathbb{H}_d$ is given by
$$d\mathcal{V}_d=\sinh^{d-1}\eta d\eta\prod_{i=1}^{d-1}(\sin\theta_i)^{d-1-i}d\theta_i$$
and then the density law of $(\eta_d(t))_{t\geq0},$ the process describing the hyperbolic distance from $\underline{O}_d$, is equal to
\begin{equation}
p_d(\eta,t)=u_{d}(\eta,t)\sinh^{d-1}\eta,\quad \eta>0,
\end{equation}
which is solution to the following partial differential equation
\begin{align}
\frac{\partial}{\partial t} p_d(\eta,t)&=\mathcal {P}^*p_d(\eta,t)\notag\\
&= \frac{1}{2\sinh^{d-1}\eta}\left[\frac{\partial^2}{\partial\eta^2}-\frac{d-1}{\tanh\eta}\frac{\partial}{\partial\eta}+\frac{d-1}{\sinh^2\eta}\right]p_d(\eta,t).
\end{align}

Furthermore, by considering the infinitesimal generator $\mathcal{P}$, it is possible to show that $(\eta_d(t))_{t\geq0}$ satisfies the following stochastic differential equation
\begin{equation}\label{eq:sdehyp}
d\eta_d(t)=\frac{d-1}{2\tanh \eta_d(t)}dt+dW(t),\quad \eta_d(0)=0
\end{equation}
with $W(t)$ representing a standard Wiener process.

\subsection{Reflecting in spheres of the Poincar\`e half space}

The reflecting hyperbolic Bessel process inside the hyperbolic $d$-dimensional sphere with radius $S$ and centre $\underline{O}_d$, is defined by
\begin{equation}
\mathcal{E}_d(t)=
\begin{cases}
\eta_d(t),& \eta_d(t)\in(0,S),\\
\frac{S^2}{ \eta_d(t)},& \eta_d(t)\in[S,\infty).
\end{cases}
\end{equation}
In the previous definition
we have used the hyperbolic counterpart of the circle inversion, that is
$$\eta\eta'=S^2$$
where $\eta$ e $\eta'$ are the hyperbolic distances from the origin of two points belonging to the same geodesic curve and $S$ represents the radius of a hyperbolic disc.

\begin{theorem}
The kernel related to $(\mathcal{E}_d(t))_{t\geq 0}$ is equal to
\begin{equation}
\overline{u}_d(\eta,t)=u_{d}(\eta,t)+u_{d}\left(S^2/\eta,t\right),\quad 0<\eta\leq S,
\end{equation}
while the density function is given by
\begin{equation}
\overline{p}_d(\eta,t)=(\sinh\eta)^{d-1} u_{d}(\eta,t)+\left(\frac{S}{\eta}\right)^2\left(\sinh\left(\frac{S^2}{\eta}\right)\right)^{d-1}u_{d}\left(\frac{S^2}{\eta},t\right),\quad 0<\eta\leq S.
\end{equation}

\end{theorem}

\begin{proof}

In order to prove the statement of this theorem, we can suitably adapt  the proof of Theorem \ref{teo:kerndenssde}.
It is not hard to verify  that the reflecting condition for the kernel $\overline{u}_d(\eta,t)$ of $\mathcal{E}_d(t)$ is satisfied, that is
$$\frac{\partial \overline{u}_d(\eta,t)}{\partial t}\Big|_{\eta=S}=0,$$
and that for the probability density $\overline{p}_d(\eta,t)$ we have that
$$\int_0^S\overline{p}_d(\eta,t)d\eta=1.$$
\end{proof}
\begin{remark}
Since equation \eqref{eq:sdehyp} holds, we can apply the Meyer-Ito formula as in the proof of Theorem \ref{itobessel}. Therefore, we conclude that $\mathcal{E}_d(t)$ is solution for the following stochastic differential equation
\begin{align*}
\mathcal{E}_d(t)
&=\int_0^t\left[{\bf 1}_{ (0,S)}(\eta_d(s))-\frac{S^2}{\eta_d^2(s)}{\bf 1}_{ [S,\infty)}(\eta_d(s))\right]d\eta_d(s)+\int_0^t\frac{S^2}{\eta_d^3(s)}{\bf 1}_{ [S,\infty)}(\eta_d(s))ds-L_t(S)\notag\\
&=\int_0^t\left\{\left[{\bf 1}_{ (0,S)}(\eta_d(s))-\frac{S^2}{\eta_d^2(s)}{\bf 1}_{ [S,\infty)}(\eta_d(s))\right]\frac{d-1}{2\tanh \eta_d(s)}+\frac{S^2}{\eta_d^3(s)}{\bf 1}_{ [S,\infty)}(\eta_d(s))\right\}ds\\
&\quad+\int_0^t\left[{\bf 1}_{ (0,S)}(\eta_d(s))-\frac{R^2}{\eta_d^2(s)}{\bf 1}_{ [S,\infty)}(\eta_d(s))\right]dW(s)-L_t(S),
\end{align*}
where
$$L_t(S)=\lim_{\varepsilon \to 0}\frac1\varepsilon\int_0^t{\bf 1}_{[S,S+\varepsilon)}(\mathcal{\eta}_d(s))ds.$$

\end{remark}

\begin{theorem}
For the reflecting hyperbolic Brownian motion we have that
\begin{align}\label{eq:pdekerhyper}
\begin{cases}
\frac{\partial }{\partial t}\overline{u}_d(\eta,t)=\mathcal{P}u_d(\eta,t)+\tilde{\mathcal{P}}u_d(S^2/\eta,t),&\\
\overline{u}_d(\eta,0)=\delta(\eta),&\\
\frac{\partial}{\partial \eta}\overline{u}_d(\eta,t)\Big|_{\eta=S}=0,&
\end{cases}
\end{align}
and
\begin{align}\label{eq:pdedenhyper}
\begin{cases}
\frac{\partial }{\partial t}\overline{p}_d(\eta,t)=\mathcal{P}^*p_d(\eta,t)+\tilde{\mathcal{P}}^*p_d(S^2/\eta,t),&\\
\overline{p}_d(\eta,0)=\delta(\eta),&\\
\frac{\partial}{\partial t}\overline{p}_d(\eta,t)\Big|_{t=0}=0,&
\end{cases}
\end{align}
where
\begin{align*}
\tilde{\mathcal{P}}=\frac12\frac{\eta^4}{S^4}\left[\left(\frac{2}{\eta}-(d-1)\frac{S^2/\eta^2}{\tanh(S^2/\eta)}\right)\frac{\partial }{\partial \eta}+\frac{\partial^2 }{\partial\eta^2}\right]
\end{align*}
and
\begin{align*}
\tilde{\mathcal{P}}^*=\frac{1}{2\sinh^{d-1}(S^2/\eta)}\left\{\frac{\eta^4}{S^4}\left[\left(\frac{2}{\eta}+(d-1)\frac{S^2/\eta^2}{\tanh(S^2/\eta)}\right)\frac{\partial }{\partial \eta}+\frac{\partial^2 }{\partial\eta^2}\right]+\frac{d-1}{\sinh^2(S^2/\eta)}\right\}.
\end{align*}
\end{theorem}
\begin{proof}

The proof of the results \eqref{eq:pdekerhyper} and \eqref{eq:pdedenhyper} follows the same steps of the proof of Theorem \ref{genou} and therefore we omit the details.
\end{proof}

\begin{remark}
Since for $d=2$ we have that
\begin{align*}
P\{\eta_2(t)>\overline{\eta}\}=2\int_{\overline{\eta}}^\infty\frac{e^{-t/4}\phi e^{-\phi^2/4t}}{\sqrt{\pi}(\sqrt{2t})^3}\sqrt{\cosh\phi-\cosh\overline{\eta}}d\phi\
\end{align*}
(see Lao and Orsingher, 2007),
the following result holds
\begin{align*}
P\{\mathcal{E}_2(t)>\overline{\eta}\}&=\int_{\overline{\eta}}^S\frac{e^{-t/4}}{\sqrt{\pi}(\sqrt{2t})^3}\sinh\eta\int_\eta^\infty\frac{\phi e^{-\phi^2/4t}}{\sqrt{\cosh\phi-\cosh\eta}}d\phi d\eta\\
&\quad+\int_{\overline{\eta}}^S\frac{e^{-t/4}}{\sqrt{\pi}(\sqrt{2t})^3}\left(S/\eta\right)^2\sinh\left(S^2/\eta\right)\int_{S^2/\eta}^\infty\frac{\phi e^{-\phi^2/4t}}{\sqrt{\cosh\phi-\cosh(S^2/\eta)}}d\phi d\eta\\
&=(S^2/\eta=w)\\
&=\int_{\overline{\eta}}^{S^2/\overline{\eta}}\frac{e^{-t/4}}{\sqrt{\pi}(\sqrt{2t})^3}\sinh\eta\int_\eta^\infty\frac{\phi e^{-\phi^2/4t}}{\sqrt{\cosh\phi-\cosh\eta}}d\phi d\eta\\
&=\int_{S^2/\overline{\eta}}^\infty\frac{e^{-t/4}\phi e^{-\phi^2/4t}}{\sqrt{\pi}(\sqrt{2t})^3}\int_{\overline{\eta}}^{S^2/\overline{\eta}}\frac{\sinh\eta}{\sqrt{\cosh\phi-\cosh\eta}} d\eta d\phi\\
&\quad+\int_{\overline{\eta}}^{S^2/\overline{\eta}}\frac{e^{-t/4}\phi e^{-\phi^2/4t}}{\sqrt{\pi}(\sqrt{2t})^3}\int_{\overline{\eta}}^{\phi}\frac{\sinh\eta}{\sqrt{\cosh\phi-\cosh\eta}} d\eta d\phi\\
&=2\int_{S^2/\overline{\eta}}^\infty\frac{e^{-t/4}\phi e^{-\phi^2/4t}}{\sqrt{\pi}(\sqrt{2t})^3}(\sqrt{\cosh\phi-\cosh\overline{\eta}}-\sqrt{\cosh\phi-\cosh(S^2/\overline{\eta})})d\phi\\
&\quad+2\int_{\overline{\eta}}^{S^2/\overline{\eta}}\frac{e^{-t/4}\phi e^{-\phi^2/4t}}{\sqrt{\pi}(\sqrt{2t})^3}\sqrt{\cosh\phi-\cosh\overline{\eta}}d\phi\\
&=2\int_{\overline{\eta}}^\infty\frac{e^{-t/4}\phi e^{-\phi^2/4t}}{\sqrt{\pi}(\sqrt{2t})^3}\sqrt{\cosh\phi-\cosh\overline{\eta}}d\phi\\
&\quad-2\int_{S^2/\overline{\eta}}^\infty\frac{e^{-t/4}\phi e^{-\phi^2/4t}}{\sqrt{\pi}(\sqrt{2t})^3}\sqrt{\cosh\phi-\cosh(S^2/\overline{\eta})}d\phi\\
&=P\{\eta_2(t)>\overline{\eta}\}-P\{\eta_2(t)>S^2/\overline{\eta}\}.
\end{align*}

Also for $d=3$, we can find the same relationship between the distribution of $\mathcal{E}_3(t)$ and $\eta_3(t)$, that is
$$P\{\mathcal{E}_3(t)>\overline{\eta}\}=P\{\eta_3(t)>\overline{\eta}\}-P\{\eta_3(t)>S^2/\overline{\eta}\}.$$
\end{remark}
\begin{remark}
By means of the Millson formula we get that
\begin{align*}
P\{\eta_{d+2}(t)>\overline{\eta}\}&=-\frac{ e^{-dt}}{2\pi}\int_{\overline{\eta}}^\infty\sinh^{d}\eta\frac{\partial}{\partial \eta}u_d(\eta,t)d\eta\\
&=(\text{since} \lim_{\eta\to \infty}u_d(\eta,t)=0)\\
&=\frac{ e^{-dt}}{2\pi}\sinh^{d}\overline\eta u_d(\overline\eta,t)+d\frac{ e^{-dt}}{2\pi}\int_{\overline\eta}^\infty\cosh\eta\sinh^{d-1}\eta u_d(\eta,t)d\eta\\
&=\frac{ e^{-dt}}{2\pi}\left[\sinh^{d}\overline\eta u_d(\overline\eta,t)+dE\left(\cosh\eta_d(t){\bf 1}_{\{\eta_d(t)>\overline\eta\}}\right)\right]
\end{align*}
for $d\geq 2$. Then we obtain that
\begin{align*}
P\{\mathcal{E}_{d+2}(t)>\overline{\eta}\}&=\int_{\overline\eta}^{S^2/\overline\eta}\sinh^{d+1}\eta u_{d+2}(\eta,t)d\eta\\
&=\frac{ e^{-dt}}{2\pi}\left[\sinh^{d}\overline\eta u_d(\overline\eta,t)-\sinh^{d}(S^2/\overline\eta) u_d(S^2/\overline\eta,t)+dE\left(\cosh\eta_d(t){\bf 1}_{\{\overline\eta<\eta_d(t)<S^2/\overline\eta\}}\right)\right].
\end{align*}
\end{remark}

\subsection{Orthogonal reflection in spheres of the Poincar\`e disc}

The hyperbolic Brownian motion is also analyzed in the unit-radius Poincar\`e disc $D=\{w=re^{i\theta},|r|<1\}$ which is an alternative Euclidean model to the planar hyperbolic space. By using the following conformal Cayley mapping
$$w=\frac{z-i}{-iz+1}=\frac{iz+1}{z+i}$$
it is possible to transform $\mathbb{H}_2=\{z=x+iy,y>0\}$ into the unit-radius disc $D$. Since 
$$x=\frac{2r\cos\theta}{r^2-2r\sin\theta+1},\quad y=\frac{1-r^2}{r^2-2r\sin\theta+1}$$
the hyperbolic laplacian in polar coordinates becomes
\begin{align}
y^2\left(\frac{\partial^2}{\partial x^2}+\frac{\partial^2}{\partial y^2}\right)=\frac{(1-r^2)^2}{2^2}\left\{\frac1r\frac{\partial}{\partial r}\left(r\frac{\partial}{\partial r}\right)+\frac{1}{r^2}\frac{\partial^2}{\partial \theta^2}\right\}
\end{align}  
(see Comtet and Monthus, 1996 or Lao and Orsingher, 2007), and then the kernel of the hyperbolic Brownian motion $(D(t))_{t\geq 0}$ inside $D$ satisfies the following equation
\begin{align*}
\frac{\partial }{\partial t}k(r,t)=\frac{(1-r^2)^2}{2^2}\left\{\frac1r\frac{\partial}{\partial r}\left(r\frac{\partial}{\partial r}\right)\right\}k(r,t)
\end{align*}
with initial condition $h(r,0)=\delta(r)$.
The stochastic differential equation solved by the process $(D(t))_{t\geq 0}$ reads
\begin{equation}
dD(t)=\frac{(1-D^2(t))^2}{4D(t)}dt+\frac{1-D^2(t)}{\sqrt{2}}dW(t).
\end{equation}
Moreover between the hyperbolic distance in the Poincar\`e upper-half plane $\eta$ and $r$ the following relationship holds
$$\eta=\log\frac{1+r^2}{1-r^2}$$
or equivalently 
$$r=\tanh \frac\eta2.$$ 
We indicate the reflecting hyperbolic Brownian motion inside a disc with radius $V<1$ by $(\mathcal{D}(t))_{t\geq 0}$ means of circular inversion. Therefore, we have that
\begin{equation}
\mathcal{D}(t)=D(t){\bf 1}_{(0,V)}(D(t))+\frac{V^2}{ D(t)}{\bf 1}_{ [V,1)}(D(t)).
\end{equation}

By means of the same arguments exploited in the previous sections we can immediately obtain the results concerning the reflecting Brownian motion in $D$.  For this reason, in what follows, we omit the details of the proofs.

 The process $\mathcal{D}(t)$ admits kernel function given by
 \begin{equation}
 \overline{k}(r,t)=k(r,t)+k\left(\frac{V^2}{r},t\right),\quad 0<r\leq V<1,
 \end{equation}
 and probability distribution
  \begin{equation}
 \overline{h}(r,t)=rk(r,t)+\frac{V^4}{r^3}k\left(\frac{V^2}{r},t\right),\quad 0<r\leq V<1,
 \end{equation}
 where
 \begin{equation}
 k(r,t)=u_2\left(\log\frac{1+r^2}{1-r^2},t\right)\frac{4r}{1-r^4}.
 \end{equation}
 
 The kernel is solution to the Cauchy problem
 \begin{align*}
\begin{cases}
\frac{\partial }{\partial t} \overline{k}(r,t)=\frac{(1-r^2)^2}{2^2}\mathcal{M}k(r,t)+\frac{(1-(V^2/r)^2)^2}{2^2}\tilde{\mathcal{M}}k\left(\frac{V^2}{r},t\right),&\\
 \overline{k}(r,0)=\delta(r),&\\
\frac{\partial}{\partial r} \overline{k}(r,t)\Big|_{r=V}=0,
\end{cases}
\end{align*}
where ${\mathcal{M}}$ and $\tilde{\mathcal{M}}$ are the operator which define the partial differential equation governing the reflecting planar Brownian motion.

Finally, by applying again the Meyer-Ito rule, we are able to prove that
\begin{align*}
\mathcal{D}(t)
&=\int_0^t\left[{\bf 1}_{ (0,V)}(D(s))-\frac{V^2}{D^2(s)}{\bf 1}_{ [V,1)}(D(s))\right]dD(s)+\int_0^t\frac{V^2}{D^3(s)}{\bf 1}_{ [V,1)}(D(s))ds-L_t(V)\notag\\
&=\int_0^t\left\{\left[{\bf 1}_{ (0,V)}(D(s))-\frac{V^2}{D^2(s)}{\bf 1}_{ [V,1)}(D(s))\right]\frac{(1-D^2(s))^2}{4D(s)}+\frac{V^2}{D^3(s)}{\bf 1}_{ [V,1)}(D(s))\right\}ds\\
&\quad+\int_0^t\left[{\bf 1}_{ (0,V)}(\eta_d(s))-\frac{V^2}{D^2(s)}{\bf 1}_{ [V,1)}(D(s))\right]\frac{1-D^2(s)}{\sqrt{2}}dW(s)-L_t(V).
\end{align*}

\small{

}

\end{document}